\documentclass[finalversion]{FPSAC2017-arxiv}

\articlenumber{62}

\usepackage{lipsum, float,enumerate}

\newtheorem{theorem}{Theorem}[section]
\newtheorem{lemma}[theorem]{Lemma}
\newtheorem{definition}[theorem]{Definition}
\newtheorem{remark}[theorem]{Remark}

\newtheorem{corollary}[theorem]{Corollary}

\numberwithin{equation}{section}
\numberwithin{figure}{section}
\numberwithin{table}{section}

\def\R{\mathbb{R}}

\def\N{\mathbb{N}}
\def\calh{\mathcal H}
\newcommand{\bp}{\ensuremath{\mathbb P}}

\DeclareMathOperator{\Id}{\ensuremath{Id}}

\newcommand\blfootnote[1]{%
  \begingroup
  \renewcommand\thefootnote{}\footnote{#1}%
  \addtocounter{footnote}{-1}%
  \endgroup
}

\title{The critical exponent: a novel graph invariant}

\author{Dominique Guillot\thanks{\href{mailto:dguillot@udel.edu}{dguillot@udel.edu}}\addressmark{1}, Apoorva
Khare\thanks{\href{mailto:khare@stanford.edu}{khare@stanford.edu}}\addressmark{2},
\and Bala Rajaratnam\thanks{\href{mailto:brajaratnam01@gmail.com}{brajaratnam01@gmail.com}}\addressmark{3}}

\address{\addressmark{1}University of Delaware, Newark, DE, USA \\
\addressmark{2}Stanford University, Stanford, CA, USA \\
\addressmark{3}University of California Davis, Davis, CA, USA}

\received{\today}

\abstract{
A surprising result of FitzGerald and Horn (1977) shows that
$A^{\circ \alpha} := (a_{ij}^\alpha)$
is positive semidefinite (p.s.d.) for every entrywise
nonnegative $n \times n$ p.s.d.~matrix $A = (a_{ij})$ if
and only if $\alpha$ is a positive integer or $\alpha \geq n-2$.
Given a graph $G$, we consider the refined problem
of characterizing the set $\mathcal{H}_G$ of entrywise
powers preserving positivity for matrices with a zero pattern encoded by 
$G$. Using algebraic and combinatorial methods, we
study how the geometry of $G$ influences the set
$\mathcal{H}_G$. Our treatment provides new and exciting connections
between combinatorics and analysis, and leads us
to introduce and compute a new graph invariant called the
\textit{critical exponent}.
}

\resume{
Un r\'esultat surprenant de FitzGerald et Horn (1977) d\'emontre que $A^{\circ \alpha} := (a_{ij}^\alpha)$ est semi-d\'efinie positive pour chaque matrice semi-d\'efinie positive de dimension $n$ avec des entr\'ees non-n\'egatives si et seulement si $\alpha$ est un entier positif ou $\alpha \geq n-2$. Pour un graph $G$ donn\'e, nous consid\'erons une g\'en\'eralization naturelle du probl\`eme en \'etudiant l'ensemble $\mathcal{H}_G$ de puissances pr\'eservant la positivit\'e des matrices ayant une structure de z\'eros encod\'ee par $G$. \`A l'aide de m\'ethodes alg\'ebriques et combinatoires, nous analysons de quelle fa\c con la g\'eometrie du graph $G$ d\'etermine l'ensemble $\mathcal{H}_G$. Notre travail fournit de nouvelles connexions excitantes entre la combinatoire et l'analyse, et nous m\`ene \`a d\'efinir et calculer un nouvel invariant que l'on nomme l'{\it exposant critique} d'un graph.
}


\begin{document}

\maketitle

\section{Introduction and main results}

Let\blfootnote{D.G., A.K., and B.R. are partially supported by the following: US Air
Force Office of Scientific Research grant award FA9550-13-1-0043, US
National Science Foundation under grant DMS-0906392, DMS-CMG 1025465,
AGS-1003823, DMS-1106642, DMS-CAREER-1352656, Defense Advanced Research
Projects Agency DARPA YFA N66001-11-1-4131, the UPS Foundation,
SMC-DBNKY, an NSERC postdoctoral fellowship, and the University of
Delaware Research Foundation.} $\N$ denote the set of positive integers.
Given $n \in \N$ and $I \subset \R$, let $\bp_n(I)$ denote the set of
symmetric positive semidefinite $n \times n$ matrices with entries in
$I$. Given two $n \times n$ matrices $A = (a_{ij})$ and $B =
(b_{ij})$, their Hadamard (or Schur, or entrywise) product, denoted by $A
\circ B$, is defined by $A \circ B  := (a_{ij} b_{ij})$. Note that $A
\circ B$ is a principal submatrix of the tensor product $A \otimes B$. As
a consequence, if $A$ and $B$ are positive semidefinite, then so is $A
\circ B$. This result is known in the literature as the \textit{Schur
product theorem}.

Given $\alpha \in \R$, we denote the entrywise $\alpha$th power of a
matrix $A$ with nonnegative entries by $A^{\circ \alpha} :=
(a_{ij}^\alpha)$, where we define $0^\alpha := 0$ for all $\alpha$.
By the Schur product theorem, $A^{\circ k}$ is positive
(semi)definite for all positive (semi)definite matrices $A$ and all $k
\in \N$. It is natural to ask if other real powers have the same
property. This problem was resolved by FitzGerald and
Horn, who revealed a surprising phase transition phenomenon.

\begin{theorem}[{FitzGerald--Horn \cite{FitzHorn}}]\label{Tcomp}
Let $n \geq 2$.
\begin{enumerate}
\item[1)] If $\alpha \in \N$ or $\alpha \geq n-2$, then
$A^{\circ \alpha} \in \bp_n(\R)$ for all $A \in \bp_n([0,\infty))$.
\item[2)] If $\alpha \in (0, n-2) \setminus \N$, then there exists a
matrix $A \in \bp_n([0,\infty))$ such that $A^{\circ \alpha} \not\in
\bp_n(\R)$.
\end{enumerate}
\end{theorem}

We now refine the above problem by restricting the set of matrices to
those with a given sparsity structure. Given a finite undirected simple
graph $G = (V,E)$ with nodes $V = \{ 1,2,\dots,n\}$, and a subset $I
\subset \R$, define
\begin{equation}
\bp_G(I) := \{ A \in \bp_n(I) : a_{ij} = 0\ \forall (i,j) \not\in E,\
i\ne j \}.
\end{equation}

\noindent We denote $\bp_G(\R)$ by $\bp_G$. All graphs below are
assumed to be finite and simple.
\vspace*{-0.45cm}
\begin{figure}[H]
\centering
\begin{minipage}{0.4\linewidth}
\begin{center}
\includegraphics[width=3cm]{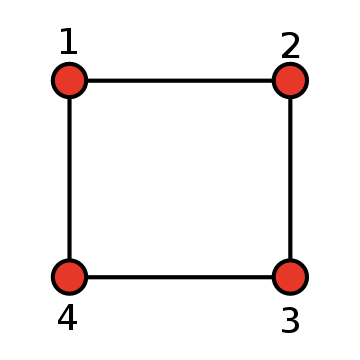}
\end{center}
\end{minipage}
\begin{minipage}{0.4\linewidth}
\begin{equation*}
\large
\begin{pmatrix}
* & * & 0 & * \\
* & * & * & 0 \\
0 & * & * & * \\
* & 0 & * & *
\end{pmatrix}
\end{equation*}
\end{minipage}
\caption{$\bp_G$ for $G$ a $4$-cycle.
Entries with an asterisk are not constrained.}
\end{figure}
\vspace*{-0.5cm}

The goal of this paper is to study the set of entrywise powers preserving
positivity on the set $\bp_G$. Such structured matrices arise naturally
in various subfields of mathematics, including combinatorial matrix
analysis \cite{Agler_et_al_88, Brualdi_mutually}, spectral graph theory
\cite{Fallat-Hogben}, and graphical models \cite{lauritzen}. As we
explain below, such matrices and their entrywise transforms are also of
importance in modern-day applications in high-dimensional
covariance estimation, making the problem at once classically motivated
as well as timely.

We now establish further notation. Note that
when $\alpha \not\in \N$, and $A$ is a real matrix,
$A^{\circ\alpha}$ is not always well-defined. 
Thus, we follow Hiai \cite{Hiai2009} and work with
the odd and even extensions to $\R$ of the power functions. Define
\begin{equation}
\psi_\alpha(x) := {\rm sgn}(x) |x|^\alpha, \qquad \phi_\alpha(x) :=
|x|^\alpha, \qquad \forall\ x \in \R \setminus \{ 0 \},
\end{equation}

\noindent and $\psi_\alpha(0) = \phi_\alpha(0)  := 0$. Given $f: \R \to
\R$, and $A = (a_{ij})$, define $f[A] := (f(a_{ij}))$. We now introduce
the main objects of study in this paper.  

\begin{definition}\label{D1}
Let $n \geq 2$ and let $G = (V,E)$ be a simple graph on $V =
\{1,\dots,n\}$. We define: 
\begin{align*}
\calh_G := &\ \{ \alpha \in \R : A^{\circ \alpha} \in \bp_G \text{ for
all } A \in \bp_G([0,\infty)) \},\\
\calh_G^\psi := &\ \{ \alpha \in \R : \psi_\alpha[A] \in \bp_G \text{ for
all } A \in \bp_G(\R) \}, \\
\calh_G^\phi := &\ \{ \alpha \in \R : \phi_\alpha[A] \in \bp_G \text{ for
all } A \in \bp_G(\R) \}.
\end{align*}
\end{definition}

\cref{Tcomp} thus shows: $\calh_{K_n} = \N \cup [n-2, \infty)$. The sets $\calh_{K_n}^\psi$
and $\calh_{K_n}^\phi$ have also been computed, and exhibit similar phase
transitions.

\begin{theorem}[FitzGerald--Horn \cite{FitzHorn}, Bhatia--Elsner
\cite{Bhatia-Elsner}, Hiai \cite{Hiai2009}, Guillot--Khare--Rajaratnam
\cite{GKR-crit-2sided}]\label{Tcomplete}
Let $n \geq 2$. The $\calh$-sets of powers preserving positivity for $G =
K_n$ are: 
\[
\calh_{K_n} = \ \N \cup [n-2,\infty), \qquad
\calh_{K_n}^\psi = \ (-1+2\N) \cup [n-2,\infty), \qquad
\calh_{K_n}^\phi = \ 2\N \cup [n-2, \infty).
\]
\end{theorem}

\noindent (See \cite{GKR-crit-2sided} for more details.)
\cref{Tcomplete} demonstrates that there is a threshold value
above which every power function $x^\alpha, \psi_\alpha$, or
$\phi_\alpha$ preserves positivity on $\bp_n([0,\infty))$ or $\bp_n(\R)$,
when applied entrywise. The threshold is commonly referred to as the
\textit{critical exponent} for preserving positivity. We now extend this
notion to all graphs.

\begin{definition}
Given a graph $G$, define the \emph{(Hadamard) critical exponents of $G$}
to be
\begin{align*}
CE_H(G) := &\ \min \{\alpha \in \R : A \in \bp_G([0,\infty)) \Rightarrow
A^{\circ \beta} \in \bp_G \textrm{ for every } \beta \geq \alpha\},\\
CE_H^\psi(G) := &\ \min \{\alpha \in \R : A \in \bp_G(\R) \Rightarrow
\psi_\alpha[A] \in \bp_G \textrm{ for every } \beta \geq \alpha\},\\
CE_H^\phi(G) := &\ \min \{\alpha \in \R : A \in \bp_G(\R) \Rightarrow
\phi_\alpha[A] \in \bp_G \textrm{ for every } \beta \geq \alpha\}.
\end{align*}
\end{definition}

These critical exponents appear to be new graph invariants,
not previously studied in the literature.
Note that since every graph $G = (V,E)$ is contained in a
complete graph, the critical exponents of $G$ are well defined by \cref{Tcomplete}, and bounded above by $|V|- 2$. However,
computing critical exponents is a challenging problem at the intersection
of graph theory, analysis, and matrix theory. Indeed, the
critical exponents are not known for many families of non-complete
graphs, and provide interesting avenues of research in combinatorial
matrix analysis, which also have the potential to impact other
areas.\smallskip

We now state our main result. To do so, we first recall the
notion of \textit{chordal graphs}. These are precisely the graphs $G$ in
which every cycle of length $4$ or more has a chord. 
Chordal graphs are prominent in mathematics as well as applications. They
are also known as decomposable graphs, triangulated graphs, and rigid
circuit graphs; have a rich structure, and include several well-known
examples of graphs (see \cref{Table_chordal}).
Chordal graphs play a fundamental role in multiple areas including the
matrix completion problem
\cite{Bala-PD-completions,Grone-PD-completions,paulsen_et_al}, maximum
likelihood estimation in the theory of Markov random fields \cite[Section
5.3]{lauritzen}, and perfect Gaussian elimination \cite{golumbic}.

Let $K_n^{(1)}$ be the complete graph on $n$ vertices with one edge
missing. Then we have:

\begin{theorem}[Main result]\label{Tmain}
Let $G$ be any chordal graph with at least $2$ vertices and let $r$ be
the largest integer such that $K_r^{(1)}$ is 
a subgraph of $G$. Then
\begin{equation}
\calh_G = \N \cup [r-2,\infty), \quad
\calh_G^\psi = (-1+2\N) \cup [r-2, \infty), \quad
\calh_G^\phi = 2\N \cup [r-2, \infty). 
\end{equation}
In particular, $CE_H(G) = CE_H^\psi(G) = CE_H^\phi(G) = r-2$. 
\end{theorem}

Our main result extends the previous work in \cref{Tcomplete} to
the important family of chordal graphs. Moreover, it shows how the
problem of finding powers preserving positivity is solvable using
combinatorial techniques.\smallskip

We conclude this section with some remarks. First, the cones $\bp_G$ of
structured matrices naturally arise in applications, as (inverse)
covariance/correlation matrices with an underlying graphical model \cite{lauritzen}.
Powering such matrices entrywise is a way to regularize them in
high-dimensional probability and statistics. This procedure often improves their
properties (e.g.~condition number) and helps separate signal from noise
-- see \cite{bickel_levina, Li_Horvath, Zhang_Horvath} for more details.
\cref{Tmain} and related results are relevant in this context.
For instance, we show in recent works
\cite{GKR-critG,Guillot_Khare_Rajaratnam2012} that the critical exponent
$CE_H(G)$ of any tree or bipartite graph $G$ is $1$. As a consequence,
unlike in the unconstrained case of $\bp_n$ (i.e., $K_n$), families of
sparse as well as dense graphs $G = (V,E)$ can have very small critical
exponents that do not grow with $V$. This is important as such small
powers can regularize matrices, yet minimally modify their entries.

This work is an extended abstract of \cite{GKR-critG}. Understanding which powers preserve
positivity is part of a broad program by the authors; 
see \cite{BGKP-fpsac, BGKP-fixeddim, BGKP-hankel, GKR-crit-2sided,
GKR-critG, Guillot_Khare_Rajaratnam2012, GKR-lowrank, Guillot_Rajaratnam2012, Guillot_Rajaratnam2012b} and the references therein. Our
work has yielded surprising connections to other areas such as Schur
polynomials and symmetric function theory; see \cite{BGKP-fpsac,
BGKP-fixeddim} for more details.

\section{Proof of the main result}\label{Sproofs}

In this section we provide the main ideas used to prove \cref{Tmain}. Complete proofs as well as other ramifications
of the results in this paper can be found in \cite{GKR-critG}.

We begin by recalling some properties of chordal graphs (see
e.g.~\cite[Chapter 5.5]{Diestel}, \cite[Chapter 4]{golumbic}). Given a
graph $G = (V,E)$, and $C \subset V$, denote by $G_C$ the subgraph of $G$
induced by $C$. A \textit{clique} in $G$ is a complete induced subgraph
of $G$. A subset $C \subset V$ \textit{separates} $A \subset V$ from $B
\subset V$ if every path from a vertex in $A$ to a vertex in $B$
intersects $C$. A partition $(A,C,B)$ of subsets of $V$ is a
\textit{decomposition} of $G$ if $G_C$ is a clique and $C$ separates $A$
from $B$ (see \cref{Fdecomp}). A graph $G$ is \textit{decomposable}
if either $G$ is complete, or if there exists a decomposition $(A,C,B)$
of $G$ such that $G_{A \cup C}$ and $G_{B \cup C}$ are decomposable.
\vspace*{-0.5cm}
\begin{figure}[H]
\centering
\includegraphics[width=5.35cm]{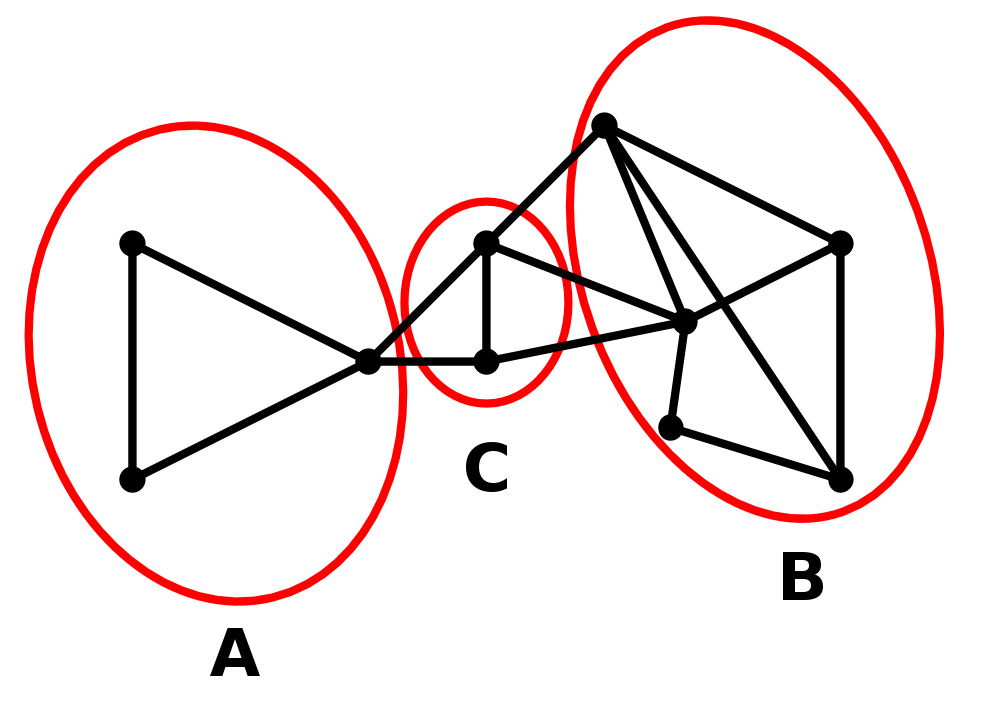}
\caption{Decomposition of a graph.}
\label{Fdecomp}
\end{figure}

Let $G$ be a graph and let $B_1, \dots, B_k$ be a sequence of subsets of
vertices of $G$. Define: 
\begin{equation}\label{Ehistories}
H_j := B_1 \cup \dots \cup B_j, \qquad R_j = B_j \setminus H_{j-1},
\qquad S_j = H_{j-1} \cap B_j, \qquad 1 \leq j \leq k, 
\end{equation}
and $H_0 := \emptyset$. The sets $H_j, R_j$, and $S_j$ are respectively
called the
\textit{histories, residuals, and separators} of the sequence. The
sequence $B_1, \dots, B_k$ is said to be a \textit{perfect ordering} if: 
\begin{enumerate}
\item For all $1 < i \leq k$, there exists $1 \leq j < i$ such that $S_i
\subset B_j$; and
\item The sets $S_i$ induce complete graphs for all $1 \leq i \leq k$. 
\end{enumerate}

Decompositions and perfect orderings provide important characterizations
of chordal graphs, as summarized in \cref{Tchordalcarac}. 

\begin{theorem}[{\cite[Chapter
2]{lauritzen}}]\label{Tchordalcarac} Given a graph $G =
(V,E)$, the following are equivalent: 
\begin{enumerate}
\item $G$ is chordal (i.e., each cycle with $4$ vertices or more in $G$
has a chord).
\item $G$ is decomposable. 
\item The maximal cliques of $G$ admit a perfect ordering. 
\end{enumerate}
\end{theorem}

We now provide a correspondence between decompositions of a
graph $G$, and decompositions of matrices in the associated cone $\bp_G$.
In the statement of the result and below, given a graph $G$ and an
induced subgraph $G'$, we identify $\bp_{G'}(I)$ with a subset of
$\bp_G(I)$ when convenient, via the assignment $M \ \mapsto \ M \ \oplus
\ {\bf 0}_{(V(G) \setminus V(G')) \times (V(G) \setminus V(G'))}$.

\begin{lemma}\label{Ldecomp_chord}
Let $G = (V,E)$ be a graph with a decomposition $(A,C,B)$ of $V$, and let
$M$ be a symmetric matrix. Assume the principal submatrices $M_{AA}$ and
$M_{BB}$ of $M$ are invertible. Then the following are equivalent: 
\begin{enumerate}
\item $M \in \bp_G$.
\item $M = M_1 + M_2$ for some matrices $M_1 \in \bp_{G_{A \cup C}}$ and
$M_2 \in \bp_{G_{B \cup C}}$. 
\end{enumerate} 
\end{lemma}

The proof of \cref{Ldecomp_chord} requires working with Schur
complements. Recall that a symmetric block matrix 
$M = \begin{pmatrix}
A & B \\
B^T & C
\end{pmatrix}$
with $C$ invertible is positive definite if and only if $C$ is positive
definite, and the Schur complement of $C$ in $M$ 
\[
M/C := A - BC^{-1}B^T
\]
is positive definite. Similarly, if $A$ is invertible, then $M$ is
positive definite if and only if $A$ is positive definite and the Schur
complement $M/A := C-B^TA^{-1}B$ is positive definite.

\begin{proof}[Proof of \cref{Ldecomp_chord}.]
Clearly $(2) \implies (1)$. Conversely, write $M \in
\bp_G$ in block form as
\begin{equation*}
M = \begin{pmatrix}
M_{AA} & M_{AC} & 0 \\
M_{AC}^T & M_{CC} & M_{CB} \\
0 & M_{CB}^T & M_{BB}
\end{pmatrix}.
\end{equation*}

\noindent Then $M = M_1 + M_2$, with
\begin{equation*}
M_1 := \begin{pmatrix}
M_{AA} & M_{AC} & 0 \\
M_{AC}^T & M_{AC}^T M_{AA}^{-1} M_{AC} & 0 \\
0 & 0 & 0
\end{pmatrix}, \qquad M_2 := \begin{pmatrix}
0 & 0 & 0 \\
0 & M_{CC} - M_{AC}^T M_{AA}^{-1} M_{AC} & M_{CB} \\
0 & M_{CB}^T & M_{BB}
\end{pmatrix}.
\end{equation*}

\noindent Using properties of Schur complements, we easily verify that
$M_1 \in \bp_{G_{A \cup C}}, M_2 \in \bp_{G_{B \cup C}}$.
\end{proof}

Note that \cref{Ldecomp_chord} also provides information about the
extreme points of the convex  cone $\bp_G$ when $G$ is decomposable. As
we show below, the problem of understanding the geometry of $\bp_G$ is
closely related to the computation of the $\calh$ sets in \cref{D1}.  

\begin{remark}
When $G$ has a decomposition $(A,C,B)$ and $M \in \bp_G$, then $M$ also
factors as 
\begin{equation}\label{Emiracle_decomposition}
M = 
\begin{pmatrix}
M_{AA} & 0 & 0 \\
M_{AC}^T & \Id_{|C|} & M_{CB} \\
0 & 0 & M_{BB}
\end{pmatrix}
\begin{pmatrix}
M_{AA}^{-1} & 0 & 0 \\
0 & S & 0\\
0 & 0 & M_{BB}^{-1}
\end{pmatrix}
\begin{pmatrix}
M_{AA} & 0 & 0 \\
M_{AC}^T & \Id_{|C|} & M_{CB} \\
0 & 0 & M_{BB}
\end{pmatrix}^T, 
\end{equation}
where $\Id_k$ denotes the $k \times k$ identity matrix, and
$S := M_{CC} - M_{AC}^T M_{AA}^{-1} M_{AC} - M_{CB} M_{BB}^{-1}
M_{CB}^T$. We will make use of this factorization later.
\end{remark}

\cref{Ldecomp_chord} provides a powerful technique to verify when
functions preserve positivity on $\bp_G$ when applied
entrywise. Indeed, suppose $G$ is a graph with a decomposition $(A,C,B)$.
Let $M \in \bp_G$. Write $M = M_1 + M_2$ as in \cref{Ldecomp_chord},
with $M_1 \in \bp_{G_{A \cup C}}$ and $M_2 \in \bp_{G_{B \cup C}}$. 
Given a real function $f: \R \to \R$, recall that we denote by $f[M]$ the
matrix $(f(m_{ij}))$. Now if $f[M] - f[M_1] - f[M_2]$ is positive
semidefinite, then $f[M] \in \bp_G$ if $f$ preserves positivity on $\bp_{G_{A \cup C}}$
and $\bp_{G_{B \cup C}}$. Thus, we introduce the following notion.

\begin{definition}
Given a graph $G$ and a
function $f : \R \to \R$ with $f(0) = 0$, we say that $f[-]$ is
\emph{Loewner super-additive on $\bp_G(\R)$} if $f[A+B] - f[A] - f[B]
\in \bp_G(\R)$ for $A,B \in \bp_G(\R)$.
\end{definition} 

Note that this notion coincides with the usual notion of super-additivity
on $[0,\infty)$ when $G$ has only one vertex.

The above discussion shows that a function preserves
positivity on $\bp_G$ if it preserves positivity on $\bp_{G_{A \cup C}}$
and $\bp_{G_{B \cup C}}$ and is Loewner super-additive on $\bp_{G_C}$.
\cref{Tchordal} below shows that the converse also holds under
certain assumptions.

\begin{theorem}\label{Tchordal}
Let $G = (V,E)$ be a graph with a decomposition $(A,C,B)$, and let $f:
\R \to \R$.
\begin{enumerate}
\item If $f[-]$ preserves positivity on $\bp_{G_{A \cup C}}$ and on
$\bp_{G_{B \cup C}}$ and is Loewner super-additive on $\bp_{G_C}$ then
$f[-]$ preserves positivity on $\bp_G$. 
\item Conversely, if $f = \psi_\alpha$ or $f= \phi_\alpha$ and $f[-]$
preserves positivity on $\bp_G$, then $f[-]$ is Loewner super-additive on
$\bp_{G_{C'}}$ for every clique $C' \subset C$ for which there exist
vertices $v_1 \in A, v_2 \in B$ that are adjacent to every $v \in C'$
(see \cref{FTchordal}).
\end{enumerate}
\end{theorem}
\vspace*{-0.25cm}
\begin{figure}[H]
\centering
\includegraphics[width=5.75cm]{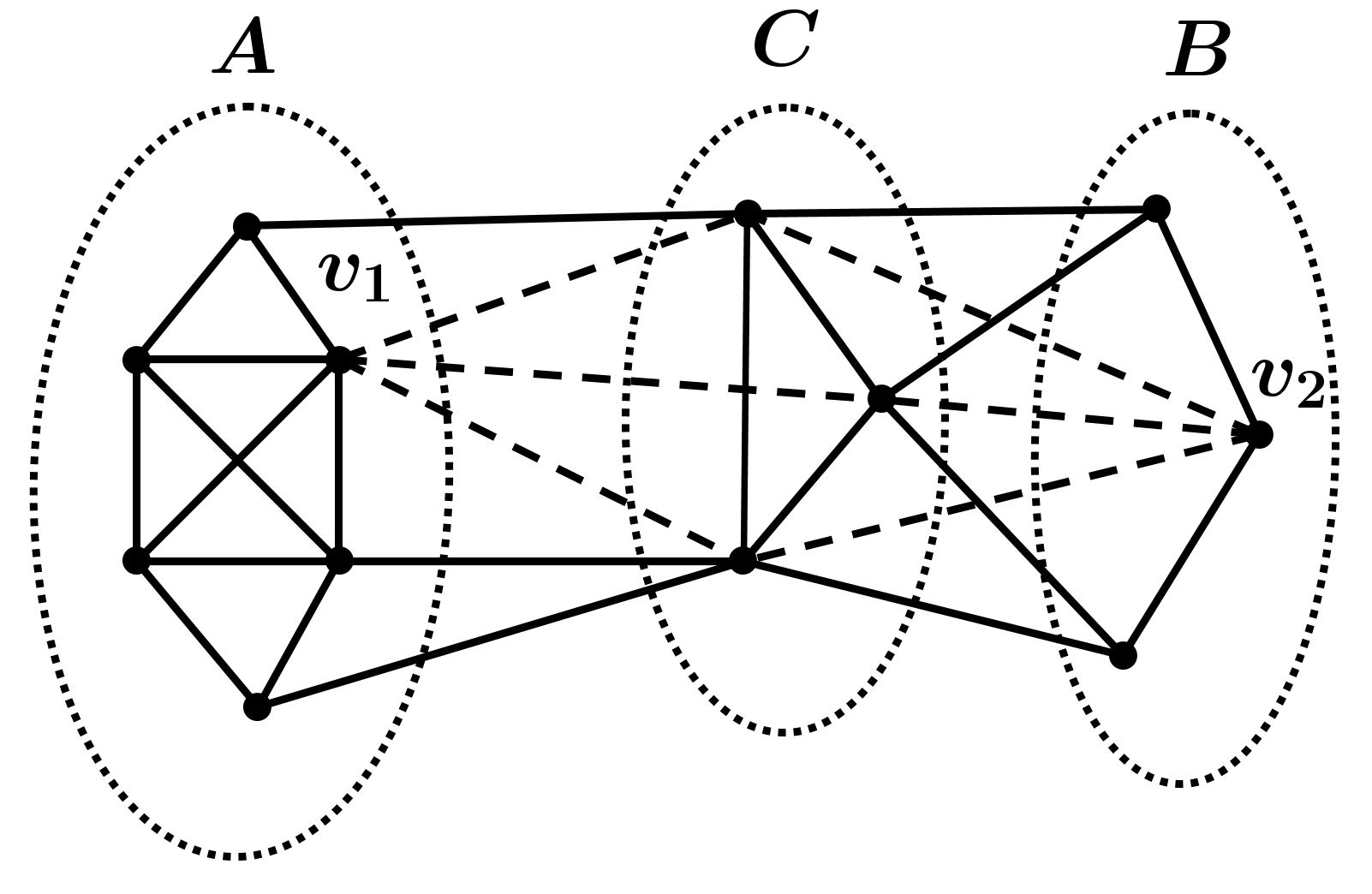}
\caption{Illustration of the second part of \cref{Tchordal}.}
\label{FTchordal}
\end{figure}

To prove \cref{Tchordal}, we recall
previous work on Loewner super-additive functions.

\begin{theorem}[{Guillot, Khare, and Rajaratnam \cite[Theorem
5.1]{GKR-crit-2sided}}]\label{Tsuperadd}
Given an integer $n \geq 2$, the sets of entrywise powers
$\alpha \in \R$, such that the functions $f_\alpha(x) = x^\alpha,
\psi_\alpha(x), \phi_\alpha(x)$ are Loewner super-additive maps on
$\bp_n(\R)$ are, respectively,
\[
\N \cup [n, \infty), \qquad
(-1+2\N) \cup [n, \infty), \qquad
2\N \cup [n, \infty).
\]
Moreover, the same results hold if $\bp_n(\R)$ is replaced by the set of
rank one matrices in $\bp_n(\R)$.
\end{theorem}

\begin{proof}[Sketch of the proof of \cref{Tchordal}]
We only prove the second part (see \cite{GKR-critG} for a complete proof). 
Suppose $f = \psi_\alpha$ or $\phi_\alpha$ for $\alpha \in \R$, and
$f[-]$ preserves positivity on $\bp_G$. Then clearly $f[-]$ preserves
positivity on $\bp_{G_{A \cup S}}$ and $\bp_{G_{B \cup S}}$. Moreover,
suppose there exist $v_1 \in A$, $v_2 \in B$, and a clique $C' \subset C$
of size $m$ such that $v_1$ and $v_2$ are adjacent to every vertex in
$C'$. Assume, without loss of generality, that the vertices of $G$ are
labelled in the following order: $v_1$, the $m$ vertices in $C'$, $v_2$,
and the remaining vertices of $G$. Now given vectors ${\bf u}, {\bf v}
\in \R^m$ and a $m \times m$ symmetric matrix $M$, define the matrix
\begin{equation}\label{Enk1}
W({\bf u}, {\bf v},M) := \begin{pmatrix} 1 & {\bf u}^T & 0\\ {\bf u} & M
& {\bf v}\\ 0 & {\bf v}^T & 1 \end{pmatrix}.
\end{equation}

\noindent Then $W({\bf u}, {\bf v}, {\bf u} {\bf u}^T + {\bf v} {\bf
v}^T) \oplus {\bf 0}_{|V| - (m+2)} \in \bp_G(\R)$, so by the
assumptions on $f$, we conclude that $f[W({\bf u}, {\bf v}, {\bf u} {\bf
u}^T + {\bf v} {\bf v}^T)] = W(f[{\bf u}], f[{\bf v}], f[{\bf u} {\bf
u}^T + {\bf v} {\bf v}^T]) \in \bp_{m+2}(\R)$. Using the factorization
\eqref{Emiracle_decomposition}, we conclude that 
\begin{equation}\label{Esuper_rank1}
f[{\bf u} {\bf u}^T + {\bf v} {\bf v}^T] - f[{\bf u}]f[{\bf u}^T] -
f[{\bf v}]f[{\bf v}^T] = f[{\bf u} {\bf u}^T + {\bf v} {\bf v}^T] -
f[{\bf uu^T}] - f[{\bf vv}^T] \geq 0.
\end{equation}

\noindent Thus $f = \psi_\alpha, \phi_\alpha$ is Loewner super-additive
on rank one matrices in $\bp_m$. By \cref{Tsuperadd},
we conclude that $f$ is Loewner super-additive on all of
$\bp_m$.
\end{proof}

We now provide a sketch of the proof of \cref{Tmain} (see \cite{GKR-critG} for full details).

\begin{proof}[Sketch of the proof of \cref{Tmain}]
Suppose $G$ is a chordal graph, which we may assume to be
connected. Let $r$ be as in the statement of the theorem.
One can show that \cref{Tcomplete} still holds when $K_n$ is
replaced by $K_n^{(1)}$. As a result, 
\begin{equation}\label{Einclusion}
\calh_G \subset \N \cup [r-2,\infty), \quad \calh_G^\psi \subset
(-1+2\N) \cup [r-2,\infty), \quad \calh_G^\phi = 2\N \cup [r-2,\infty).
\end{equation}

\noindent We now prove the reverse inclusions. By \cref{Tchordalcarac}, the maximal cliques of $G$ admit a perfect ordering
$\{C_1, \dots, C_k\}$. We will prove the reverse inclusions in
\eqref{Einclusion} by induction on $k$. If $k=1$, then $G$ is complete
and the inclusions clearly hold by \cref{Tcomplete}. Suppose the
result holds for all chordal graphs with $k=l$ maximal cliques, and let
$G$ be a graph with $k = l+1$ maximal cliques. For $1 \leq j \leq k$,
define 
\begin{equation}\label{EhistoriesC}
H_j := C_1 \cup \dots \cup C_j, \qquad C_j = C_j \setminus H_{j-1},
\qquad S_j = H_{j-1} \cap C_j
\end{equation}

\noindent as in \eqref{Ehistories}. By \cite[Lemma 2.11]{lauritzen}, the
triplet $(H_{k-1}, S_k, R_k)$ is a decomposition of $G$. Let $\alpha \in
[r-2,\infty)$. By the induction hypothesis, the three $\alpha$th power
functions preserve positivity on $\bp_{G_{H_{k-1} \cup S_k}} =
\bp_{G_{H_{k-1}} }$. Moreover, since $\alpha \geq r-2$, they also
preserve positivity on $\bp_{G_{C_k \cup S_k}} = \bp_{G_{C_k}}$. We now
claim that $r \geq |S_k| + 2$. Clearly, $|S_k| \leq r$ since $S_k$ is
complete. If $|S_k| = r$, then $C_k$ is contained in one of the previous
cliques, which is a contradiction.
Suppose instead that $|S_k| = r-1$. Since $\{C_1, \dots, C_k\}$ is a
perfect ordering, $S_k \subset C_i$ for some $i < k$. Let $v \in C_i
\setminus S_k$ and let $w \in R_k$. As $v, w$
are adjacent to every $s \in S_k$, the subgraph of $G$ induced by $S_k
\cup \{v,w\}$ is isomorphic to $K_{r+1}^{(1)}$, which contradicts the
definition of $r$. It follows that $r \geq |S_k| + 2$, as
claimed. 
Now by \cref{Tsuperadd},
the $\alpha$th power functions are Loewner super-additive on
$\bp_{S_k}$. Applying \cref{Tchordal}, we conclude that $\alpha
\in \calh_G^{\psi}, \calh_G^\phi$, and hence $\alpha \in \calh_G$. This
concludes the proof of the theorem. 
\end{proof}

The following corollary shows how to systematically compute
the critical exponent of a chordal graph. 

\begin{corollary}\label{Cformula}
Suppose $G = (V,E)$ is chordal, $V = \{ v_1, \dots, v_m
\}$, and denote the maximal cliques in $G$ by $C_1, \dots,
C_n$. Define the ``maximal clique matrix'' of $G$ to be $M(G) := ({\bf 1}(v_i \in C_j))$. Then the
critical exponent of $G$ equals the largest entry of $M(G)^T M(G) -
2 \Id_{|V|}$, i.e.,
\begin{equation}\label{Eformula}
CE_H(G) = CE^\psi_H(G) = CE^\phi_H(G) = \max_{i,j} (u_i^T u_j - 2
\delta_{i,j}),
\end{equation}

\noindent where $u_1, \dots, u_n \in \{ 0, 1 \}^m$
are the columns of $M(G)$.
\end{corollary}

In particular, \cref{Cformula} can be used to compute the critical exponent of interval graphs, which are a well-known class of chordal graphs. Note that \cref{Tchordal} can be used to compute the critical exponents of several other important graphs; see \cref{Table_chordal}.

\begin{table}[H]
\centering
\begin{tabular}{|c|c|}
\hline
Graph $G$ & $CE_H(G),\ CE_H^\psi(G),\ CE_H^\phi(G)$ \\ \hline
Tree & 1 \\
Complete graph $K_n$ & $n-2$ \\
Minimal planar triangulation of $C_n$ for $n \geq 4$ & 2 \\
Apollonian graph, $n \geq 3$ & $\min(3,n-2)$ \\
Maximal outerplanar graph, $n \geq 3$ & $\min(2,n-2)$\\
Band graph with bandwidth $d \leq n$ & $\min(d,n-2)$\\
Split graph with maximal clique $C$ & $\max(|C|-2, \max \deg(V \setminus
C))$\\ \hline
\end{tabular}
\medskip
\caption{Critical exponents of important families of chordal graphs with
$n$ vertices.}
\label{Table_chordal}
\end{table}

\section{Non-chordal graphs: results and open problems}

Computing the set of powers preserving positivity on $\bp_G$ for general
non-chordal graphs $G$ still remains open. In this section, we mention
recent results along this direction, and conclude by outlining several
open questions.

As shown in \cref{Sproofs}, decompositions of graphs can be used
to make reductions when computing critical exponents of chordal graphs.
For non-chordal graphs, the decomposition process can still be iterated
until components cannot be decomposed anymore. The resulting components
are called the \textit{prime components} of the graphs. The following
result is akin to \cref{Tchordal} for non-chordal graphs. 

\begin{theorem}[{\cite[Theorem 4.1]{GKR-critG}}]\label{Tprime}
Let $G$ be a graph with a perfect ordering $\{B_1, \dots, B_k\}$ of its
prime components, and let $f : \R \to \R$ be such that $f(0) = 0$. Define $s := \max_{i=1, \dots, k} |S_i|$, 
where $S_i$ is defined as in \eqref{Ehistories}. If $f[-]$ preserves
positivity on $\bp_{B_i}$ for all $1 \leq i \leq k$ and is Loewner
super-additive on $\bp_{K_s}$, then $f[-]$ preserves positivity on
$\bp_G$.
\end{theorem}

\subsection{Cycles and bipartite graphs}

Chordal graphs are graphs without induced cycles of length $4$ or more. A
next natural step is thus to examine the case of cycles. In
recent work \cite{GKR-critG}, we show:

\begin{theorem}[{\cite[Proposition 4.3]{GKR-critG}}]\label{Tcycle}
For all $n \geq 3$, 
$\calh_{C_n} = \calh_{C_n}^\psi = [1,\infty), \textrm{ and }
\calh^\phi_{C_4} = [2,\infty)$.
Moreover, for $n>4$, $[2,\infty)
\subset \calh_{C_n}^\phi \subset [1,\infty)$, with $1 \notin
\calh_{C_n}^\phi$ for $n$ even.
\end{theorem}

Note that \cref{Tcycle} is in line with \cref{Tmain} as
$r=3$ is the biggest integer such that $K_r^{(1)}$ is contained
in $C_n$. 

Another very common family of non-chordal graphs is the bipartite graphs. 

\begin{theorem}\label{Tbipartite}
Suppose $G$ is a connected bipartite graph with at least $3$ vertices.
Then,
\[
\calh_G = [1,\infty), \qquad [2,\infty) \subset
\calh^\phi_G \subset [1,\infty),
\qquad \{ 1 \} \cup [3,\infty) \subset \calh^\psi_G \subset [1,\infty).
\]

\noindent If moreover $K_{2,2} \subset G \subset K_{2,m}$ for some $m
\geq 2$, then 
\[
\calh^\phi_G = [2,\infty), \qquad \{ 1 \} \cup [2,\infty)
\subset \calh^\psi_G \subset [1,\infty).
\]
\end{theorem}

Akin to \cref{Tcycle}, note that \cref{Tbipartite} is also in
agreement with \cref{Tmain}. Moreover, \cref{Tbipartite} has a surprising conclusion: broad families of
dense graphs such as complete bipartite graphs can have small critical
exponents that do not grow with the number of vertices.
This has important applications in high-dimensional statistics:
for appropriate structures of zeros, small powers can be used to
minimally modify the entries of covariance matrices to
improve their properties (see Introduction), while maintaining
positivity. Note also that the result is in sharp contrast to the general
case (\cref{Tcomp}), where there is no underlying structure of
zeros.

\subsection{Concluding remarks and open problems}

The critical exponent of several other graphs (including coalescences of
graphs and graphs obtained by pasting cycles to other graphs) were
computed in \cite{GKR-critG}. However, the critical
exponent is unknown for general graphs; it appears that new ideas in
algebra, combinatorics, and convex geometry will be required to 
solve the question completely. We conclude this short paper by
formulating some open problems that we hope will stimulate research at
the intersection of these three areas.\medskip

\pagebreak
\noindent \textbf{Open problems:}
\begin{enumerate}
\item[1)] Every graph $G$ for which $CE_H(G)$ is currently known
satisfies $CE_H(G) = r-2$ where $r$ is the largest integer such that $G$
contains $K_r^{(1)}$ as a subgraph. Does this equality
in fact hold for every graph?
\item[2)] It appears that the critical exponent of a graph is always an
integer. Can this be proved directly without explicitly computing
critical exponents?
\item[3)] If $G'$ is obtained from $G$ by adding a new vertex to $G$, and
connecting it to every vertex of $G$, is it true that $CE_H(G') \leq
CE_H(G) + 1$?
\end{enumerate}




%
%



\end{document}